\documentclass[11pt]{amsart}
\usepackage{
 amsmath, 
 amsxtra, 
 amsthm, 
 amssymb, 
 etex, 
 mathrsfs, 
 mathtools, 
 tikz-cd, 
 bbm,
 xr,
 comment,
 enumitem}
\usepackage[toc]{appendix} 
\usepackage[all]{xy}
\usepackage{hyperref}
\usepackage{url}
\usepackage{tipa}
\usepackage{dsfont}
\usepackage{ textcomp }
\usepackage{stmaryrd} 
\usepackage{amssymb}
\usepackage{mathabx} 
\usepackage{amsmath} 
\usepackage{amscd}
\usepackage{amsbsy}
\usepackage{comment, enumerate}
\usepackage{color}
\usepackage{mathtools,caption}
\usepackage{tikz-cd}
\usepackage{longtable}
\usepackage[utf8]{inputenc}
\usepackage[OT2,T1]{fontenc}
\usepackage{changepage}
\usepackage{commath,enumerate}

\newtheorem{theorem}{Theorem}[section]
\newtheorem*{theorem*}{Theorem}
\newtheorem{lemma}[theorem]{Lemma}

\newtheorem{proposition}[theorem]{Proposition}
\newtheorem{corollary}[theorem]{Corollary}

\newtheorem{defn}[theorem]{Definition}

\numberwithin{equation}{section}
\newtheorem{lthm}{Theorem} 

\theoremstyle{remark}
\newtheorem{remark}[theorem]{Remark}

\newcommand{\TT}{\mathbf{T}}
\newcommand{\R}{\mathbb{R}}
\newcommand\EatDot[1]{}

\newcommand{\Fp}{\FF_p}

\newcommand{\RR}{{\mathbb R}}
\newcommand{\cS}{{\mathcal{S}}}
\newcommand{\cX}{\mathcal{X}}

\newcommand{\ZZ}{\mathbb{Z}}
\newcommand{\FF}{\mathbb{F}}

\definecolor{Green}{rgb}{0.0, 0.5, 0.0}

\newcommand{\Z}{\mathbb{Z}}

\newcommand{\Fq}{\mathbb{F}_q}

\newcommand{\Div}{\mathrm{Div}}
\newcommand{\Divz}{\mathrm{Div}^0}

  \DeclareFontFamily{U}{wncy}{}
  \DeclareFontShape{U}{wncy}{m}{n}{<->wncyr10}{}
  \DeclareSymbolFont{mcy}{U}{wncy}{m}{n}
  \DeclareMathSymbol{\sha}{\mathord}{mcy}{"58}
  \DeclareMathSymbol{\zhe}{\mathord}{mcy}{"11}
\title[Zeta functions, isogeny graphs and modular curves]{On the zeta functions of supersingular isogeny graphs and modular curves}

\makeatletter
\let\@wraptoccontribs\wraptoccontribs
\makeatother

\usepackage{tikz,xcolor,hyperref}

\author[A. Lei]{Antonio Lei}
\address[Lei]{Department of Mathematics and Statistics\\University of Ottawa\\
150 Louis-Pasteur Pvt\\
Ottawa, ON\\
Canada K1N 6N5}
\email{antonio.lei@uottawa.ca}

\author[K. Müller]{Katharina Müller}
\address[Müller]{D\'epartement de Math\'ematiques et de Statistique\\
Universit\'e Laval, Pavillion Alexandre-Vachon\\
1045 Avenue de la M\'edecine\\
Qu\'ebec, QC\\
Canada G1V 0A6}
\email{katharina.mueller.1@ulaval.ca}


\begin{document}
\begin{abstract}
Let $p$ and $q$ be distinct prime numbers, with $q\equiv 1\pmod{12}$. Let $N$ be a positive integer that is coprime to $pq$. We prove a formula relating the Hasse--Weil zeta function of the modular curve $X_0(qN)_{\Fq}$ to the Ihara zeta function of the  $p$-isogeny graph of supersingular elliptic curves defined over $\overline{\mathbb{F}_q}$ equipped with a $\Gamma_0(N)$-level structure.  When $N=1$, this recovers a result of Sugiyama.
\end{abstract}

\maketitle

\section{Introduction}
\label{S: Intro}
Zeta functions hold significant importance across various realms of number theory. They serve as powerful tools, which encode intricate arithmetic information of  mathematical objects. In this article, we study relations between two families of such zeta functions, namely the Hasse--Weil zeta functions attached to modular curves and the Ihara zeta functions attached to supersingular isogeny graphs.

Given an algebraic curve $C$ defined over a finite field $k$, the Hasse--Weil zeta function attached to $C$ encodes the number of rational points on $C$ defined over a finite extension of $k$ (see \cite[\S1.5]{li-book} for a detailed discussion). We are particularly interested in the case where $C$ is a modular curve. More specifically, let $p$ and $q$ be two distinct prime numbers with $q\equiv 1\pmod{12}$. Let $X_0(q)$ denote the modular curve classifying isomorphism classes of elliptic curves equipped with a $\Gamma_0(q)$-level structure. Let $W(X_0(q)_{\Fp},S)\in 1+S\ZZ[[S]]$ be the Hasse--Weil zeta function attached to $X_0(q)_{\Fp}$ (see Definition~\ref{def:Hase-weil-zeta} for a precise definition).

In graph theory, the Ihara zeta function is defined using prime closed geodesics. See \cite[Chapter~2]{terras} for a comprehensive survey. These functions are related to the adjacency matrix and the valency matrix of a graph. Furthermore, analogous to the analytic class number formula for its number field counterpart, the Ihara zeta function can be used to compute the size of the Jacobian of a graph. In the present article, we are interested in understanding connections between the Hasse--Weil zeta function coming from modular curves to the Ihara zeta functions of supersingular isogeny graphs, whose definition we review below.

Let $\Sigma=\{E_1,\dots, E_n\}$  denote a set of representatives of isomorphism classes of supersingular elliptic curves defined over $\overline{\Fq}$, where $n=\frac{q-1}{12}$. We define a graph $X_p^q(1)$ whose set of vertices is $\Sigma$ and the set of edges is given by $p$-isogenies (see Definition~\ref{def:graph}). Let $Z(X_p^q(1),S)\in\ZZ[S]$ be the Ihara zeta function of this graph (see Definition~\ref{def:ihara-zeta} for a precise definition).

Sugiyama showed in \cite[Thereom~1.1]{sugiyama} that the two zeta functions discussed above are related by the following explicit equation: 
\begin{equation}
    W(X_0(q)_{\Fp},S)Z(X_p^q(1),S)=\frac{1}{(1-S)^2(1-pS)^2(1-S^2)^{\frac{n(p-1)}{2}}}
\end{equation}
(note that we have replaced the symbol $N$ in loc. cit. by $q$ here; in the present article, $N$ will denote a positive integer that is not necessarily a prime number).
We remark that this relation has also been observed by Li in \cite[P.54]{li-book}.

The goal of this article is to generalize this result to more general levels. More specifically, we replace $X_0(q)_{\Fp}$ by $X_0(qN)_{\Fp}$, where  $N$ is a positive integer coprime to $pq$. Correspondingly, the graph $X_p^q(1)$ shall be replaced by the graph whose vertices are the isomorphism classes of pairs $(E,C)$, where $E\in\Sigma$ and $C$ is a cyclic subgroup of order $N$ in $E$, while the edges are still given by $p$-isogenies.  
\begin{lthm}[Corollary \ref{cor:comparison-zeta-function}]\label{thmA}
    The following equality holds:
    \[W(X_0(qN)_{\Fp},S)W(X_0(N)_{\Fp},S)^{-2}Z(X_p^q(N),S)=(1-S^2)^{\chi(X_p^q(N))},\]
where $\chi(X_p^q(N))$ denotes the Euler characteristic of the graph $X_p^q(N)$.\end{lthm}
We recover Sugiyama's main result from \cite{sugiyama} on noting that $\chi(X_p^q(1))=-\frac{n(p-1)}{2}$ and $W(X_0(1),S)=\frac{1}{(1-S)(1-pS)}$.

We briefly describe the proof of Theorem~\ref{thmA}. Let $\Divz(X_p^q(N))$ be the group of zero divisors on the vertices of the graph $X_p^q(N)$. Let $T$ be the Hecke algebra acting on $\Divz(X_p^q(N))$ and let $\mathbf{T}$ be the Hecke algebra acting on the space of $q$-new cuspforms of weight $2$ and level $qN$ (see Definitions~\ref{def:hecke-graph} and \ref{def:hecke-forms}).  We make use of a result of Ribet \cite{ribet90}, which says that these two Hecke algebras are isomorphic to deduce an isomorphism of $T\otimes\R$-modules $$\Divz(X_p^q(N))\otimes \RR\cong S_2(\Gamma_0(qN))_{q-\textup{new}}$$
(see Proposition~\ref{prop:iso-S2-Div0}). This generalizes the corresponding result for $N=1$ in 
 \cite[Proposition~3.2]{sugiyama}. The aforementioned isomorphism allows us to relate the Brandt matrix (see Definition~\ref{def:brandt}) to the Ihara zeta function. To conclude the proof, we relate the Brandt matrix to the Hasse--Weil zeta function, which can be described using the Fourier coefficients of cuspforms of level $qN$.

\subsection*{Acknowledgement}
  We thank the anonymous referee for very helpful comments and suggestions. The authors' research is supported by the NSERC Discovery Grants Program RGPIN-2020-04259 and RGPAS-2020-00096.

\section{On the Ihara zeta function of a supersingular isogeny graph}

The goal of this section is to give an explicit formula of the Ihara zeta function of a supersingular isogeny graph. The main result of the present section is  Corollary~\ref{cor:zeta-function-graph}.
\subsection{Defining supersingular isogeny graphs}

Let $p$ and $q$  be two distinct  prime numbers, and $N$  a nonnegative integer coprime to $pq$. Assume that $q\equiv 1\pmod {12}$. Let $B$ be a quaternion algebra that is only ramified at $\infty$ and $q$. Let $R$ be a fixed maximal order in $B$ and let $I_1,\dots,I_n$ be fixed representatives for the ideal classes in $R$. 
For $1\le i\le n$ let $R_i$ be the right order of $I_i$. There are $n$ distinct isomorphism classes $\Sigma=\{E_1,\dots, E_n\}$ of supersingular elliptic curves defined over $\overline{\mathbb{F}_q}$ such that $\textup{End}(E_i)=R_i$. As $q\equiv 1\pmod{12}$, it follows furthermore that $R_i^\times=\{\pm 1\}$ for $1\le i\le n$ and that $n=\frac{q-1}{12}$ (see the discussion in \cite[\S3.1]{sugiyama}, bearing in mind that the symbol $N$ in loc. cit. is replaced by  $q$ here).

\begin{defn}\label{def:graph}
We define an undirected graph $X_p^q(N)$ whose set of vertices is given by
\[V(X_p^q(N))=\{(E,C)\mid E\in \Sigma, C\subset E[N] \textup{ a cyclic subgroup of order $N$}\}.\] 
We draw an edge between $(E,C)$ and $(E',C')$ whenever there is a $p$-isogeny $\phi\colon E\to E'$  such that $\phi(C)=C'$ (loops are allowed). 

We denote by $\Div(X_p^q(N))$ and $\Divz(X_p^q(N))$ the divisors and zero divisors of $X_p^q(N)$ over $\Z$, respectively. 
\end{defn}

\subsection{Modular curves and Hecke algebras}
Let $X_0(qN)$ be the modular curve of level $\Gamma_0(qN)$. It classifies isomorphism classes of pairs $(E,C)$, where $E$ is an elliptic curve and $C$ is a cyclic subgroup of order $qN$ in $E$. The curve $X_0(qN)_{\Fq}$ consists of two copies of $X_0(N)_{\Fq}$ intersecting at supersingular points. 

We shall relate $\Divz(X_p^q(N))$ to the space of weight two $q$-newforms of level $qN$, which we introduce below.

\begin{defn}
    We write $S_2(\Gamma_0(qN))$ for the $\RR$-vector space of weight-two cuspforms of level $\Gamma_0(qN)$. Analogously we write $S_2(\Gamma_0(N))$ for the $\RR$-vector space of weigh-two cuspforms of level $\Gamma_0(N)$. 
\end{defn}
We have two natural embeddings
\[\iota_1\colon S_2(\Gamma_0(N))\to S_2(\Gamma_0(qN)),\quad f(z)\mapsto f(z)\]
and 
\[\iota_2\colon S_2(\Gamma_0(N))\to S_2(\Gamma_0(qN)), \quad f(z)\mapsto f(qz).\]
\begin{defn}
    We call the space $$\iota_1(S_2(\Gamma_0(N))\oplus \iota_2(S_2(\Gamma_0(N))\subset S_2(\Gamma_0(qN))$$ the \textbf{$N$-old space} of $S_2(\Gamma_0(qN))$, denoted by $S_2(\Gamma_0(qN))_{q-\textup{old}}$. We define the \textbf{$N$-new space} $S_2(\Gamma_0(qN))_{q-\textup{new}}$ to be the orthogonal complement of $S_2(\Gamma_0(qN))$ with respect to the Petersson inner product. 
\end{defn}

We now introduce the definition of Hecke operators.
\begin{defn}
Let $\ell$ be a prime number that is coprime to $qN$. We define the action of the Hecke correspondence $T_\ell$ on $X_0(qN)$ by sending $(E,C)$ to
\[T_\ell(E,C)=\sum_D(E/D,C+D/D),\]
where the sum runs over all cyclic subgroups $D$ of $E$ of order $\ell$.
If $\ell\mid N$, we define
\[T_\ell(E,C)=\sum_D (E/D,C+D/D),\]
where the sum runs over all cyclic subgroups of order $\ell$ not intersecting $C$. \end{defn}

As $q$ is coprime to $N$, we can decompose every $qN$ level structure into a product $C\times C_q$, where $C$ is of level $N$ and $C_q$ is of level $q$.
Let $w_q$ be the Atkin--Lehner involution on $X_0(qN)$ defined by sending $(E,C,C_q)$ to $(E/C_q,C+C_q/C_q,E[q]/C_q)$. It turns out that $T_q$ acts as $-w_q$ on the toric part of $X_0(qN)$ and that $w_q$ acts as the Frobenius on the supersingular points \cite[propositions 3.7 and 3.8]{ribet90}.

Recall that the vertices of $X_p^q(N)$ are tuples $(E,C)$ of supersingular elliptic curves $E$ and cyclic subgroups of order $N$. Thus, $T_\ell$ acts  on $\Div(X_p^q(N))$ and $\Divz(X_p^q(N))$.

\begin{defn}
\label{def:hecke-graph}
    Let $T$ be the $\ZZ$ algebra generated by all the operators $T_\ell$ as operators on $\Divz(X_p^q(N))$.
\end{defn}

We remark that the Hecke operator $T_\ell$ preserves both $S_2(\Gamma_0(qN))_{q-\textup{old}}$ and $S_2(\Gamma_0(qN))_{q-\textup{new}}$. This allows us to give the following:
\begin{defn}
\label{def:hecke-forms}
Let $\mathbf{T}$ (resp. $\mathbf{T}'$) be the subalgebra of $\textup{End}(S_2(\Gamma_0(qN))_{q-\textup{new}}$ (resp. $\textup{End}(S_2(\Gamma_0(qN))$) generated by the Hecke operators $T_\ell$ as $\ell$ runs through all  prime numbers.
\end{defn}

As both algebras $T$ and $\mathbf{T}$ are generated by the Hecke operators $T_\ell$, there is a natural map  $T\to \mathbf{T}'\to \mathbf{T}$.
We recall the following result of Ribet. 
\begin{theorem}
\label{thm:iso-hecke-algebras}
    The Hecke algebras $T$ and $\mathbf{T}$ isomorphic. In particular,  an element in $t\in T$ acts trivially on $\textup{Div}^0(X_l^q(N))$ if and only if it has trivial image in $\mathbf{T}$.
\end{theorem}
\begin{proof}
See \cite[Theorem 3.10]{ribet90}.
\end{proof}

\subsection{Relation between the zero divisor group and $q$-newforms}

We now prove the key technical ingredient of the proof of Theorem~\ref{thmA}, where we relate the zero divisor of the graph $X^q_p(N)$ to $S_2(\Gamma_0(qN))_{q\textup{-new}}$.

In what follows, we shall regard $S_2(\Gamma_0(qN))_{q\textup{-new}}$ as a $T$-module after identifying $T$ with $\TT$ via Theorem~\ref{thm:iso-hecke-algebras}. 
\begin{proposition}\label{prop:iso-S2-Div0}
There is an isomorphism of $T\otimes \RR$-modules
\[S_2(\Gamma_0(qN))_{q\textup{-new}}\cong\Divz(X_p^q(N))\otimes \mathbb{R}.\footnote{
    Note that the left-hand side of the isomorphism does not depend on the prime $p$. On the right-hand side, while the prime $p$ appears in the notation, it is in fact independent of $p$. This is because the divisor group is defined in terms of the set of vertices of the graph $X_p^q(N)$, which is independent of $p$. The prime $p$ is only relevant when we define the edges of the graph.}\]
\end{proposition}

\begin{proof}
Let $T_0\subset T$ be the subalgebra generated by the Hecke operators $T_\ell$ with $(\ell,qN)=1$. The Hecke operators $T_\ell$ with $(\ell,qN)=1$ are represented by commuting symmetric matrices. Let $\cS$ be the set of $\R$-valued characters  on $T_0$. Then 
 \[\Divz(X_p^q(N))\otimes \mathbb{R}=\bigoplus_{\gamma\in\cS} V(\gamma),\]
 where $V(\gamma)$ is a $T_0\otimes \RR$-submodule of $\Divz(X_p^q(N))\otimes \mathbb{R}$ on which $T_0$ acts via $\gamma$. A priori, this is only a decomposition of $T_0\otimes \RR$-modules. Since $T$ is commutative, $V(\gamma)$ is invariant under the action of all elements of $T$. Therefore, the aforementioned decomposition is in fact a decomposition of  $T\otimes\RR$-modules.

Note that $S_2(\Gamma_0(qN))_{q\textup{-new}}$ can equally be decomposed into  $\gamma$-eigenspaces $W(\gamma)$. As the Hecke algebras $T$ and $\mathbf{T}$ are isomorphic by Theorem~\ref{thm:iso-hecke-algebras}, we may decompose $S_2(\Gamma_0(qN))_{q\textup{-new}}$ into submodules on which $T_0$ acts via $\gamma$ as $\gamma$ runs over $\cS$.

Let $f$ be a normalized newform of level $M'\mid qN$ and let $W(f)$ be the subspace of $S_2(\Gamma_0(qN))$ generated by $\{f(dz)\mid d\mid (qN)/M'\}$. It is well known that there is a decomposition
\[S_2(\Gamma_0(qN))=\bigoplus_f W(f),\]
where the sum runs over all newforms of level $M'\mid qN$. Note that \[S_2(\Gamma_0(qN))=S_2(\Gamma_0(qN))_{q-\textup{old}}\oplus S_2(\Gamma_0(qN))_{q-\textup{new}}\] is a decomposition of $\mathbf{T}$-modules. The multiplicity one theorem for cusp forms now implies that for each character $\gamma\in \cS$, there exists a unique normalized cuspform $f_\gamma$ of level $qM$ dividing $qN$ such that $T_\ell f_\gamma=\gamma(T_\ell)f_\gamma$ for all $\ell\nmid qN$.  Let $\mathbf{T}'_\gamma$ be the subalgebra of $\textup{End}(S_2(\Gamma_0(qM))$ generated by all Hecke operators. Then there is an extension $\gamma'$ of $\gamma$ such that $T_\ell f_\gamma=\gamma'(T_\ell)f_\gamma$ for all $T_\ell\in \mathbf{T}'_\gamma$. By an abuse of notation, we will denote $\gamma'$ by $\gamma$ from now on. 

Let $\mathbf{T}_\gamma$ be the Hecke algebra in $\textup{End}(S_2(\Gamma_0(qN))$ generated by all Hecke operators $T_\ell$ such that $(\ell, N/M)=1$. Let $\ell$ be a  prime number and write $\ell^k$ for the exact power of $\ell$ dividing $N/M$. Let $W(\gamma)$ be the space of modular forms generated over $\R$ by 
\[\{f_\gamma(dz): d\mid (N/M)\}.\] 
We consider two cases.

\noindent\textbf{\underline{Case 1 - $\ell\nmid M$:}} Let
\[A_\ell=\begin{pmatrix}0&0&\dots &\dots&0\\
1&0&\dots &\dots&0\\
0&1&0&\dots &0\\
\dots&\dots&\dots&\dots&\dots\\
\dots&\dots&\dots&\dots&\dots\\
\dots&\dots&1&0&-l\\
\dots&\dots&\dots&1&\gamma(T_\ell)
\end{pmatrix}\in \textup{Mat}_{k+1,k+1}(\mathbb{R}).\]
This describes the action of $T_l$ on the basis $\{f_\gamma(d\ell^iz):i=0,1,\ldots, k\}$, where $d$ is a positive integer coprime to $\ell$ dividing $N/M$.
In particular, $T_\ell$ acts by a block matrix on $W(\gamma)$, whose blocks are given by $A_\ell$. Note that the characteristic polynomial and the minimal polynomial of $A_\ell$ coincide and are of degree $k+1$. Let  $p_\ell(x)$ denote this polynomial. We obtain an isomorphism of $\RR [T_\ell]$-modules \[\mathbb{R}[T_l\mid W(\gamma)]\cong \mathbb{R}[x]/p_\ell(x).\] 

\noindent\textbf{\underline{Case 2 - $\ell| M$:}} Define
\[A_\ell=\begin{pmatrix}0&0&\dots &\dots&0\\
1&0&\dots &\dots&0\\
0&1&0&\dots &0\\
\dots&\dots&\dots&\dots&\dots\\
\dots&\dots&\dots&\dots&\dots\\
\dots&\dots&1&0&0\\
\dots&\dots&\dots&1&\gamma(T_l)
\end{pmatrix}\in \textup{Mat}_{k+1,k+1}(\mathbb{R}).\]
Once again, let  $p_\ell(x)$ be the characteristic polynomial. Then  
\[\mathbb{R}[T_\ell\mid{W(\gamma)}]\cong \mathbb{R}[x]/p_\ell(x).\]

Let $\ell_1,\dots, \ell_s$ be the primes dividing $N/M$. Let $I$ be the ideal generated $\{p_{\ell_i}(x_i), 1\le i \le s\}$ in $\R[x_1,\dots, x_s]$. There is a $T$-module isomorphism
\[W(\gamma)\cong \mathbb{R}[x_1,\dots, x_s]/I,\]
where $\mathbf{T}_\gamma$ acts on both sides via $\gamma$. 

For every character $\gamma\in\cS$, let $T_\gamma$ be the quotient of $T$ acting on $V(\gamma)$ faithfully. It follows that
\[T\otimes \RR\cong \bigoplus T_\gamma\otimes \RR\]
and we have a similar decomposition for $\mathbf{T}\otimes \RR$. It follows from Theorem~\ref{thm:iso-hecke-algebras} that $T_\gamma\otimes \RR$ and $\mathbf{T}_\gamma\otimes \RR$ are isomorphic as $\R$-algebras.  In particular, $p_\ell$ is the minimal polynomial of $T_\ell$ as an element in $T_\gamma$. It follows that $\dim(V(\gamma))\ge \dim(W(\gamma))$. This holds for all characters $\gamma\in \cS$. Furthermore, the $\RR$-vector spaces $S_2(\Gamma_0(qN))_{q\textup{-new}}$ and $\Divz(X_p^q(N))\otimes\RR$ have the same dimension (see \cite[proof of Theorem 3.10]{ribet90}). Thus, $\dim_\RR(V(\gamma))=\dim_\RR(W(\gamma))$ for all $\gamma$. This concludes the proof of the proposition.
\end{proof}

\subsection{The Brandt matrix and Ihara zeta function}
We recall the definition of the Ihara zeta function:
 \begin{defn}
\label{def:ihara-zeta}
    Given a graph $X$, we write $\chi(X)$ for its Euler characteristic. The Ihara zeta function of $X$ is defined to be $$Z(X,S)=\frac{(1-S^2)^{\chi(X)}}{\det(I-AS+(D-I)S^2)}\in \ZZ[S],$$
    where $A$ and $D$ are the adjacency matrix and valency matrix of $X$, respectively.
\end{defn}
Our goal is to describe the Ihara zeta function of $X_p^q(N)$ in terms of the Brandt matrix, which we describe below.
 
 Recall that $n=(q-1)/12$. Let $d_N$ be the number of cyclic subgroups of order $N$ in $\Z/N\Z\times \Z/N\Z$. Then the graph $X_p^q(N)$ has $nd_N=(q-1)d_N/12$ vertices.

\begin{defn}\label{def:brandt}
   Let $\{(E_i,C_i):1\le i\le nd_N\}$ denote the set of vertices of $X_p^q(N)$. Let $\cX_{i,j}^{(p)}$ the set of  $p$-isogenies from $E_i$ to $E_j$ that map $C_i$ to $C_j$. The Brandt matrix $B_p^q(N)=(b_{i,j})_{1\le i,j\le nd_N}$ is defined by $$b_{i,j}=\frac{1}{2}\left\vert\cX_{i,j}^{(p)}\right\vert.$$ 
\end{defn}

\begin{lemma} 
\label{lem-adjacency-matrix}
The Brandt matrix $B_p^q(N)$ is the adjacency matrix of the graph $X_p^q(N)$. In particular, $B_p^q(N)$ represents the adjacency operator on $\textup{Div}(X_p^q(N))$. 
\end{lemma}
\begin{proof}
    This essentially follows from definitions;  we give the details of the proof for the convenience of the reader. Let $\phi \colon E_i\to E_j$ be a isogeny of degree $p$. Then $\ker(\phi)$ is a cyclic subgroup of $E_i[p]$ of order $p$. If conversely $D\subset E_i[p]$ is a cyclic subgroup of order $p$ such that $E_i/E_i[p]\cong E_j$, then $\phi\colon E_i\to E_i/E_i[p]$ defines an isogeny of degree $p$. We obtain a well-defined surjective map
    \begin{align*}
        \kappa\colon& \{p\text{-isogenies from }E_i\text{ to }E_j\}\\
        &\to \{\textup{cyclic subgroups of order $p$ in $E_i[p]$ with $E_i/E_i[p]\cong E_j$}\}.
    \end{align*}
    Let $r\in R_i^\times =\textup{End}(E_i)^\times=\{\pm 1\}$. Clearly, $\phi\circ r$ and $\phi$ have the same kernel. Furthermore, $\phi\circ r(C_i)=\phi(C_i)$ for all cyclic subgroups $C\subset E_i[p]$. Thus, $\kappa$ is a two-to-one map. In particular, there are $b_{i,j}$ cyclic subgroups $D\subset E_i[p]$ of order $p$ such that $(E_i/D,C_i+D/D)=(E_j,C_j)$. This concludes the proof of the lemma.
\end{proof}
We conclude the present section with the following:
\begin{corollary}\label{cor:zeta-function-graph}
    The equality
    \[
    Z(X_p^q(N),S)=\frac{(1-S^2)^{\chi(X_p^q(N))}}{\det(1-B_p^q(N)S+pS^2I)}
    \]
    holds.
\end{corollary}
\begin{proof}

As $p$ is coprime to $qN$, each  $(E,C)$ admits $p+1$ isogenies of degree $p$. In particular, the graph $X_p^q(N)$ is $(p+1)$-regular.  Thus, $D-I=p I$. The corollary now follows from  Lemma~\ref{lem-adjacency-matrix}.
\end{proof}
We have a tautological exact sequence
\[0\to \Divz(X_p^q(N))\to \Div(X_p^q(N))\to \ZZ\to 0.\]
After tensoring with $\RR$ this sequence splits and we can find an element $\delta\in \Div(X_p^q(N))$ such that 
\[\Div(X_p^q(N))\otimes \mathbb{R}=(\Divz(X_p^q(N))\otimes \mathbb{R})\oplus \mathbb{R}\delta\]
as $T\otimes \RR$-modules.
For all $\ell\nmid qN$, we have $T_\ell\delta=(\ell+1)\delta$. Thus, Proposition~\ref{prop:iso-S2-Div0} implies that
\begin{equation}
 \label{eq:determinats}   
\det(1-B_p^q(N)S-lS^2)=(1-S)(1-pS)\det(1-T_pS+Sp^2\mid S_2(\Gamma_0(Nq))_{q\textup{-new}}).\end{equation}
 \section{Proof of Theorem~\ref{thmA}} 
The goal of this section is to prove Theorem~\ref{thmA} given in the introduction. As before, $p$ is a fixed prime that is coprime to $qN$. We consider $X_0(qN)_{\Fp}$, the modular curve $X_0(qN)$ as a curve over ${\Fp}$. Since $p$ is fixed throughout, we shall drop the subscript $\Fp$ from the notation for simplicity and simply write $X_0(qN)$ and $X_0(N)$ for the curves defined over $\Fp$.

The final step of our proof of Theorem~\ref{thmA} is to relate the Brandt matrix to the Hasse--Weil Zeta function, whose definition we recall below. 
\begin{defn}
\label{def:Hase-weil-zeta}
    Given an algebraic curve $C$ over $\Fp$, we define the Hasse--Weil zeta function of $C$ by $$W(C,S)=\prod_{x\in|C|}\frac{1}{1-S^{\deg(x)}}\in 1+S\ZZ[[S]],$$
    where $|C|$ is the set of closed points in $C$.  
\end{defn}

\begin{remark}\label{rk:level-one}
    If $N=1$, we have $X_0(N)=\mathbf{P}^1$. In this case, the Hasse--Weil zeta function of $X_0(N)$ is given by
\[W(X_0(N),S)=\frac{1}{(1-S)(1-pS)}.\]
\end{remark}

\begin{lemma}\label{lem:Brandt-Weils}
Let $B_p^q(N)$ be the Brandt matrix given in Definition~\ref{def:brandt}. 
Then
\[\det(1-B_p^q(N)S-pS^2)=W(X_0(qN),S)W(X_0(N),S)^{-2}.\]
\end{lemma}
\begin{proof}

As discussed in \cite[page 12]{li-book},  we can write
\[W(X_0(qN),S)=(1-S)^{-1}(1-pS)^{-1}\prod_{i=1}^{g(qN)}(1-\lambda_i(p)S+pS^2),\]
where $g(qN)$ is the genus of $X_0(qN)$ and $\lambda_i$ are eigenvalues of $T_p$ on $X_0(qN)$ counted with multiplicities. Therefore, dividing by $W(X_0(N),S)^2$ gives
\begin{align*}&
    W(X_0(qN),S)W(X_0(N),S)^{-2}\\&=(1-S)(1-pS)\prod_f(1-a_p(f)S+pS^2),
\end{align*}
where the product runs over the set of normalized eigen-$q$-newforms $f$ (counted with multiplicities) in $S_2(\Gamma_0(qN))$. 
On fixing a $T_p$-eigen-basis, we have
\[\prod_f(1-a_p(f)S+pS^2)=\det(1-ST_p+pS^2\mid S_2(\Gamma_0(qN))_{q\textup{-new}}).\]
 Theorem~\ref{thm:iso-hecke-algebras} tells us that the right-hand side is equal to \[\det(1-ST_p+pS^2\mid \Divz(X_p^q(N))\otimes \RR).\] Thus, \eqref{eq:determinats} implies that
\[\prod_f(1-a_p(f)S+pS^2)(1-S)(1-pS)=\det(1-B_p^q(N)S+pS^2),\]
from which  the result follows.
\end{proof}

We can now conclude the proof of  Theorem~\ref{thmA}:
\begin{corollary}
\label{cor:comparison-zeta-function}We have
\[W(X_0(qN),S)W(X_0(N),S)^{-2}Z(X_p^q(N),S)=(1-S^2)^{\chi(X_p^q(N))}.\]    
\end{corollary}
\begin{proof}
This follows from combining Lemma~\ref{lem:Brandt-Weils} with Corollary~\ref{cor:zeta-function-graph}.   \end{proof}

\bibliographystyle{amsalpha}
\bibliography{references}

\end{document}